\documentclass{amsart}
\usepackage{graphicx}
\usepackage{float}
\usepackage{amsthm}
\usepackage{longtable}
\newtheorem{thm}{Theorem}[section]

\newtheorem{lem}[thm]{Lemma}
\newtheorem{prop}[thm]{Proposition}
\theoremstyle{definition}

\theoremstyle{remark}

\numberwithin{equation}{section}
\newtheorem*{theorem}{\textbf{Main Theorem}}

\begin{document}

\title[ Characterization of Nilpotent Lie Algebras by Their Multiplier]{ Characterization of finite dimensional nilpotent Lie algebras by the dimension of their Schur multipliers, $s(L)=5$}
\author[A. Shamsaki]{Afsaneh Shamsaki}
\address{School of Mathematics and Computer Science\\
Damghan University, Damghan, Iran}
\email{Shamsaki.Afsaneh@yahoo.com}
\author[P. Niroomand]{Peyman Niroomand}
\email{niroomand@du.ac.ir, p$\_$niroomand@yahoo.com}
\address{School of Mathematics and Computer Science\\
Damghan University, Damghan, Iran}


\begin{abstract}

It is known that the dimension of the Schur multiplier of a non-abelian nilpotent Lie algebra $L$ of dimension
$n$  is equal to $\frac{1}{2}(n-1)(n-2)+1-s(L)$ for some $ s(L)\geq0 $. The structure of all nilpotent Lie algebras has been given for $ s(L) \leq 4 $ in several papers. Here, we are going to give the structure of all non-abelian nilpotent Lie algebras for $s(L)=5$.
 \end{abstract}
\maketitle
\section{Introduction and motivation}
Let $L$ be a finite dimensional nilpotent Lie algebra such that $L\cong F/R$ for a free Lie algebra $F$. 
Then by \cite{ba},  the Schur multiplier $\mathcal{M}(L)$ of $L$ is isomorphic to $R\cap F^2/[R,F]$.
By a result of Moneyhun in \cite{16}, there exists a non-negetive integer $t(L) $ such that 
$\dim \mathcal{M}(L)=\frac{1}{2}n(n-1)-t(L).$ 
It is a classical question to determine the structure of $L$ by looking at the dimension of  its Schur multiplier.
The answer to this problem was given  for  $t(L)\leq 8$ in \cite{2,13,12} and 
by putting some conditions on $L$ for $t(L)\leq 16$ in \cite{bo} by several authors.

From \cite{17}, when $L$ is a non-abelian nilpotent Lie algebra, the dimension of the Shur multiplier of $L$ is equal to $\frac{1}{2}(n-1)(n-2)+1-s(L) $ for some $ s(L)\geq 0 $. It not only improves the bound of Moneyhun but also let us ask the same natural question about the characterization of Lie algebras in term of size $s(L)$.
The answer to this question was given by several papers in \cite{25,29,29-} for $ s(L)\leq 4$ and for $ s(L)\leq 15$
when conditions are put on $L$ in \cite{nin}.

On the other hand, looking for instance \cite{25} shows that the characterization of 
 nilpotent Lie algebras by looking $s(L)$  causes to  classification of nilpotent Lie algebras in terms  of $ t(L) $ by a simple and a shorter way.
This paper is devoted to obtain the structure of all nilpotent Lie algebras $L$ for $t(L)=5$.

Throughout the paper, we may assume that $ L $ is a Lie algebra over an algebraically closed field of characteristic not equal to 2 and $A(n)$ and $H(m)$ are used to denote the abelian Lie algebra of dimension $n$ and the Heisenberg Lie algebra of dimension
$2m + 1$, respectively. 

For the sake of convenience  of reader some notations and terminology  from \cite{7,13,12,15}  are listed below.
\begin{center}
\begin{tabular}{lp{8cm}}
$L_{3,2}\cong H(1) $ &with a basis $\{x_{1}, x_{2}, x_{3}\}$ and the  multiplication $ [x_{1},x_{2}]$ $=x_{3} $,\\
$L_{4,3} \cong L(3,4,1,4)$ &with a basis $\{x_{1}, ..., x_{4}\}$ and the  multiplication $ [x_{1},x_{2}]$ $=x_{3}, [x_{1},x_{3}] =x_{4} $,\\
$L_{5,5}\cong L(4,5,1,6)$  &with a basis $\{ x_{1}, ..., x_{5}\}$ and the multiplication $[x_{1},x_{2}]$ $=x_{3}, [x_{1},x_{3}] =x_{5}, [x_{2},x_{4}]=x_{5} $,\\

$L_{5,6}\cong L^{\prime}(7,5,1,7)$& with a basis $ \{x_{1}, ..., x_{5}\}$  and the multiplication  $[x_{1},x_{2}]$ $=x_{3}, [x_{1},x_{3}] =x_{4}, [x_{1},x_{4}]=[x_{2},x_{3}]=x_{5} $,\\
$L_{5,7}\cong L(7,5,1,7) $ & with a basis $\{x_{1}, ..., x_{5}\}$  and the  multiplication $[x_{1},x_{2}]$ $=x_{3}, [x_{1},x_{3}] =x_{4}, [x_{1},x_{4}]=x_{5} $,\\
$ L_{5,8} \cong L(4,5,2,4)$ &with a basis $\{x_{1}, ..., x_{5}\}$  and the multiplication $[x_{1},x_{2}]=x_{4},[x_{1},x_{3}]=x_{5}$,\\
$ L_{5,9}\cong L(7,5,2,7)$ &with a basis $ \{ x_{1}, ..., x_{5}\} $ and the  multiplication  $[x_{1},x_{2}]$ $=x_{3}, [x_{1},x_{3}] =x_{4}, [x_{2},x_{3}]=x_{5} $,\\
$ L_{6,10} $ &with a basis $\{x_{1}, ..., x_{6}\}$  and the multiplication $[x_{1},x_{2}]=x_{3},[x_{1},x_{3}]=x_{6}, [x_{4},x_{5}]=x_{6}$,\\
$ L_{6,22}(\varepsilon)$ &with a basis $\{x_{1}, ..., x_{6}\}$ and the multiplication $[x_{1},x_{2}]=x_{5},[x_{1},x_{3}]=x_{6}, [x_{2},x_{4}]= \varepsilon x_{6},[x_{3},x_{4}]=x_{5}, \varepsilon \in \mathbb{F}$,\\

  $ L_{1}\cong 27B $ &with a basis $ \{x_{1}, ..., x_{7}\}$ and the multiplication $ [x_{1},x_{2}]=[x_{3},x_{4}]=x_{6},[x_{1},x_{5}]=[x_{2},x_{3}]=x_{7}$,\\
 $ L_{2}\cong 27A$ &with a basis $\{x_{1}, ..., x_{7}\}$ and the multiplication $[x_{1},x_{2}]=x_{6},[x_{1},x_{4}]=x_{7},[x_{3},x_{5}]=x_{7}$,\\
 $ 157 $ &with a basis $\{x_{1}, ..., x_{7}\}$ and the multiplication $[x_{1},x_{2}]=x_{3},[x_{1},x_{3}]=[x_{2},x_{4}]=[x_{5},x_{6}]=x_{7}$\\
  $ 37B $ & with a basis $\{x_{1}, ..., x_{7}\}$ and the multiplication $[x_{1},x_{2}] =x_{5}, [x_{2},x_{3}] =x_{6}, [x_{3},x_{4}] =x_{7}$,\\
  $ 37C $ &with a basis $\{x_{1}, ..., x_{7}\}$ and the multiplication $[x_{1},x_{2}] =[x_{3},x_{4}]=x_{5}, [x_{2},x_{3}] =x_{6}, [x_{2},x_{4}] =x_{7}$,\\
    $ 37D $ &with a basis $\{x_{1}, ..., x_{7}\}$ and the multiplication $[x_{1},x_{2}] =[x_{3},x_{4}]=x_{5}, [x_{1},x_{3}] =x_{6}, [x_{2},x_{4}] =x_{7}$.
\end{tabular}
\end{center}
 
\begin{theorem}
{\em  Let $L$ be a non-abelian $n$-dimensional nilpotent Lie algebra. Then $ s(L)=5 $  if and only if $L$  is isomorphic to one of the Lie algebras $L(4,5,2,4)\oplus A(4)$, $L(3,4,1,4)\oplus A(3)$, $L(4,5,1,6)\oplus A(2)$,  $L_{6,22}(\varepsilon)\oplus A(2)$, $ L_{6,26}\oplus A(1) ,$ $ L_{6,10} $, $L_{6,23}$,  $ L_{6,25} $, $ 37B $, $ 37C $ or  $ 37D $.}
 \end{theorem}

We state some results without proof and refer the reader to see \cite{24, 25, 27-,29-}.

  \begin{prop}$($See   \cite[Proposition 2.10]{26}$)$  \label{15}
The Schur multiplier of Lie algebras $L_{6,22}(\varepsilon)$, $L_{5,8}$, $ L_{1} $ and $L_{2}$ are abelian Lie algebras of dimension $8,6,9$ and $10$, respectively.
\end{prop}
A Lie algebra $ L $ is called capable provided that $ L\cong H/Z(H) $ for a Lie algebra $ H $. From \cite[Definition 1.4]{28}, $  Z^{*}(L) $ is used to the  denote the epicenter of $L$. The importance of  $  Z^{*}(L) $ is due to the fact that $ L $ is capable if and only if $  Z^{*}(L)=0 $. Another notion having relation to the capability is the concept of the exterior center of a Lie algebra $ Z^{\wedge}(L) $ which is introduced in \cite{24}. It is known that from \cite[Lemma 3.1]{24}$)$, $ Z^{*}(L)=Z^{\wedge}(L) $.
 \begin{lem}$($See \cite[Corollary  2.3]{29-}$)$ \label{23}
    Let $L$ be a  non-capable nilpotent  Lie algebra of dimension $n$ such that $ \dim L^{2}\geq 2 $. Then
\begin{equation*}
n-3<s(L).
\end{equation*}

\end{lem}Let $\otimes_{mod}$ be used to denote the operator of usual tensor product of Lie algebras. Then
\begin{thm}$($See \cite[Theorem 2.1]{N}$)$ \label{16}
Let $ L $ be a finite dimensional nilpotent Lie algebra non-abelian Lie algebra of class two. Then 
\begin{equation*}
0\rightarrow \ker g  \rightarrow L^{2}\otimes_{mod}L^{ab}\xrightarrow{ \; g \;} \mathcal{M}(L) \rightarrow \mathcal{M}(L^{ab})\rightarrow L^{2} \rightarrow 0
\end{equation*}
is exact, in where \\
$ g : x\otimes (z+L^{2}) \in  L^{2}\otimes_{mod}L^{ab} \mapsto [\overline{x},\overline{z}]+[R, F] \in \mathcal{M}(L)=R\cap F^{2}/[R, F]$, $ \pi (\overline{x}+R)=x $ and $ \pi (\overline{z}+R)=z$. Moreover, the subalgebra $ K=\langle [x, y]\otimes (z+L^{2})+[z,x]\otimes (y+L^{2})+[y,z]\otimes (x+L^{2})\mid x,y,z \in L\rangle $ is contained in $\ker g $.
\end{thm}
\section{the proof of main theorem}
We begin with the following lemma that is easily proven.
\begin{lem}\label{pn}
 There is no $ n$-dimensional nilpotent Lie algebra with $ s(L)=5 $, when
 \begin{itemize}
  \item[(i)] $ \dim L^{2}\geq 4 $;
 \item[(ii)]$ \dim L^{2}=1 $.
 \end{itemize}
 \end{lem}
 \begin{proof}
 \item[(i)]Let $ L $ be a nilpotent  Lie algebra such that $m=\dim L^{2} \geq  4$ and $s(L)=5$.  \cite[Theorem 3.1]{25} and our assumption imply that
 \begin{equation*}
  \frac{1}{2}(n-1)(n-2)-4= \dim \mathcal{M}(L)\leq \frac{1}{2}(n+m-2)(n-m-1)+1 \leq \frac{1}{2}(n+2)(n-5)+1.
 \end{equation*}
 It is a contradiction.
 \item[(ii)] By contrary. Let $L$ be a Lie algebra such that $ \dim L^{2}=1 $ and $ s(L)=5 $. Then by using  \cite[Lemma 3.3]{17}, $ L\cong H(m)\oplus A(n-2m-1)$ for some $m\geq 1$. Looking  \cite[Corollary 2.5]{29} shows that $s(L)=0$ or $s(L)=2$, when  $m=1$ or $m\geq 2$,  respectively. It is a contradiction. Hence the result follows.
 \end{proof}

By using Lemma \ref{pn}, we may assume that a  nilpotent Lie algebras $L$ with $s(L)=5$ has $2\leq\dim L^2\leq 3$.
First assume that $\dim L^2=2$.
 \begin{lem}\label{1*}
Let $ L $ be an $n$-dimensional non-capable nilpotent Lie algebra of dimension at most $7$ and $ \dim L^{2}=2 $. Then $ L$ is isomorphic to one of the Lie algebras $L_{6,10} $,  $ L_{2} $ or $ 157 $. Moreover,  $ s(L_{6,10})=5   $ and  $ s(L_2)=s(157)=  6 $.
\begin{proof}
 The proof is similar to \cite[Theorem 2.6]{29-}. 
\end{proof}
\end{lem}
 \begin{thm}
 Let $L$ be an $n$-dimensional nilpotent Lie algebra with $s(L)=5$ and $\dim L^{2}=2$. Then $L$ is isomorphic to one of the Lie algebras $L(4,5,2,4)\oplus A(4), L(3,4,1,4)\oplus A(3)$, $L(4,5,1,6)\oplus A(2)$,  $L_{6,22}(\varepsilon)\oplus A(2)$ or $ L_{6,10} $.
 \begin{proof}
 Sine $\dim L^2=2$, $L$ is nilpotent of class two or three. Let $ L $ be a Lie algebra of nilpotency class two. If $ L $  is a capable Lie algebra, then it should be isomorphic to one of the Lie algebras $ L_{6,22}(\varepsilon)\oplus A $, $ L_{5,8}\oplus A $ or $ L_{1}\oplus A $,  for an abelian Lie algebra $ A $ by using \cite[Corollary 2.13]{26}.

 Case (i). Let $L\cong L_{6,22}(\varepsilon)\oplus A$.  Proposition \ref{15} implies $\dim M(L_{6,22}(\varepsilon))=8 $. Since $5=s(L)=\frac{1}{2}(n-1)(n-2)+1-\dim \mathcal{M}(L)$ and $\dim \mathcal{M}(L)=8+\frac{1}{2}(n-6)(n+1)$ by using \cite[Theorem 1]{2} and \cite[Lemma 23]{16},  we have $n=8$. Hence $ L\cong L_{6,22}(\varepsilon)\oplus A(2) $.\\
 Case (ii). Let now $L\cong L_{5,8}\oplus A$. We know from Proposition \ref{15}, $\dim\mathcal{M}(L_{5,8})=6$. Since
 $5=s(L)=\frac{1}{2}(n-1)(n-2)+1-\dim \mathcal{M}(L)$ and $\dim \mathcal{M}(L)=6+\frac{1}{2}(n-5)n $ by using \cite[Theorem 1]{2} and \cite[Lemma 23]{16},  we have  $n=9 $. Therefore  $L\cong L_{5,8}\oplus A(4)\cong L(4,5,2,4)\oplus A(4)$.\\
 Case (iii). Let $L\cong L_{1}\oplus A$. We know $\dim \mathcal{M}(L_{1})=9$ by using Proposition \ref{15}. Since
  $5=s(L)=\frac{1}{2}(n-1)(n-2)+1-\dim \mathcal{M}(L)$ and $\dim \mathcal{M}(L)=9+\frac{1}{2}(n-7)(n+2) $ by using \cite[Theorem 1]{2} and \cite[Lemma 23]{16}, we have  $n=5 $,  which is contradiction. Thus $ L $ cannot be isomorphic to $ L_{1}\oplus A$. \\
Now let $L$ be  a Lie algebra of nilpotency class $ 3 $. If $ L $ is a capable Lie algebra, then it should be isomorphic to one of the Lie algebras $ L_{4,3}\oplus A(n-4) $  or $  L_{5,5}\oplus A(n-5)  $ by using  \cite[Theorem 5.5]{27}. 

Case (i). Let   $ L \cong L_{4,3}\oplus A(n-4) $. Since $ \dim \mathcal{M}(L_{4,3})=2 $ by using  \cite[Section 2]{13}, we have  $\dim \mathcal{M}(L)=2+ \frac{1}{2}(n-4)(n-1)$ by using \cite[Theorem 1]{2} and \cite[Lemma 23]{16}. Since $5=s(L)=\frac{1}{2}(n-1)(n-2)+1-\dim \mathcal{M}(L)$ and  $\dim \mathcal{M}(L)=2+ \frac{1}{2}(n-4)(n-1)$, we have $ n=7 $. Hence $ L\cong L_{4,3}\oplus A(3)\cong L(3,4,1,4)\oplus A(3)$.

Case (ii). Suppose $ L\cong  L_{5,5}\oplus A(n-5)$. \cite[Section 3]{13} shows that $ \dim \mathcal{M}(L_{5,5})=4 $. Now \cite[Theorem 1]{2} and \cite[Lemma 23]{16} imply that $\dim \mathcal{M}(L)=4+ \frac{1}{2}n(n-5)$. Since $5=s(L)=\frac{1}{2}(n-1)(n-2)+1-\dim \mathcal{M}(L)$ and  $\dim \mathcal{M}(L)=4+ \frac{1}{2}n(n-5)$,   we have $ n=7 $. Therefore $ L\cong  L_{5,5}\oplus A(2) \cong L(4,5,1,6)\oplus A(2)$.\\
If  $L$ is a non-capable Lie algebra of nilpotency class $ 2 $ or $ 3 $, then by using Lemma \ref{23}, we have $n\leq 7$. Therefore $ L\cong L_{6,10} $ by  using Lemma \ref{1*}. This completes the  proof.\\
 \end{proof}
 \end{thm}
 We now consider the case that $\dim L^2=3$. By looking all nilpotent Lie algberas listed in \cite{7}, we may  choose all $ n $-dimensional nilpotent Lie algebras $ L $ such that $ \dim L^{2}=3 $  for $ n=5 $ or $ 6 $ in the Table 1.
  \begin{longtable}{cccc}
\multicolumn{2}{c}{\textbf{Table 1. } }\\
\hline \multicolumn{1}{c}{\textbf{Name}} & \multicolumn{1}{c}{\textsf{Nonzero multiplication}}    \\
\hline
\endhead
\hline \multicolumn{2}{r}{\small \itshape continued on the next page}
\endfoot
\endlastfoot
$ L_{5,6} $ &$[x_{1},x_{2}] =x_{3}, [x_{1},x_{3}] =x_{4}, [x_{1},x_{4}] =[x_{2},x_{3}] =x_{5}$\\  \\
 $ L_{5,7} $& $[x_{1},x_{2}] =x_{3}, [x_{1},x_{3}] =x_{4}, [x_{1},x_{4}]  =x_{5}$ \\ \\
$ L_{5,9} $& $[x_{1},x_{2}] =x_{3}, [x_{1},x_{3}] =x_{4}, [x_{2},x_{3}] =x_{5}$ \\ \\
$ L_{6,6}$&$[x_{1},x_{2}] =x_{3}, [x_{1},x_{3}] =x_{4}, [x_{1},x_{4}] =[x_{2},x_{3}] =x_{5}$\\ \\
$ L_{6,7} $& $[x_{1},x_{2}] =x_{3}, [x_{1},x_{3}] =x_{4}, [x_{1},x_{4}]  =x_{5}$\\ \\
$ L_{6,9} $& $[x_{1},x_{2}] =x_{3}, [x_{1},x_{3}] =x_{4}, [x_{2},x_{3}] =x_{5}$\\ \\
$ L_{6,11}$ &$[x_{1},x_{2}] =x_{3}, [x_{1},x_{3}] =x_{4},$ \\ 
&$[x_{1},x_{4}] =[x_{2},x_{3}] =[x_{2},x_{5}]=x_{6}$ \\ 
\\
 $ L_{6,12} $& $[x_{1},x_{2}] =x_{3}, [x_{1},x_{3}] =x_{4}, [x_{1},x_{4}] =[x_{2},x_{5}] =x_{6}$  \\
  \\
   $ L_{6,13} $& $[x_{1},x_{2}] =x_{3}, [x_{1},x_{3}] = [x_{2},x_{4}] =x_{5},$ \\&$[x_{1},x_{5}] = [x_{3},x_{4}] =x_{6}$ \\
 \\
 $ L_{6,19}(\epsilon) $ & $[x_{1},x_{2}] =x_{4}, [x_{1},x_{3}] =  x_{5}, [x_{1},x_{5}] =[x_{2},x_{4}] =x_{6}, $\\ 
 &$[x_{3},x_{5}] =\epsilon x_{6}$   \\
 \\
$ L_{6,20} $ & $[x_{1},x_{2}] =x_{4}, [x_{1},x_{3}]  =x_{5},[x_{1},x_{5}] = [x_{2},x_{4}] =x_{6}$ \\ \\
 $ L_{6,23} $&$[x_{1},x_{2}] =x_{3}, [x_{1},x_{3}] =[x_{2},x_{4}] =x_{5}, [x_{1},x_{4}] =x_{6}$\\
 \\
 $ L_{6,24}(\epsilon) $ &$[x_{1},x_{2}] =x_{3}, [x_{1},x_{3}] =[x_{2},x_{4}] =x_{5},$  \\
 & $[x_{1},x_{4}] =\varepsilon x_{6}, [x_{2},x_{3}] =x_{6}$\\
 \\
 $ L_{6,25} $& $[x_{1},x_{2}] =x_{3}, [x_{1},x_{3}] =x_{5}, [x_{1},x_{4}] =x_{6}$   \\
 \\
$ L_{6,26} $& $[x_{1},x_{2}] =x_{4}, [x_{1},x_{3}] =x_{5}, [x_{2},x_{3}] =x_{6}$   \\

 \hline
\end{longtable}
 Assume $ L $ is nilpotent Lie algebra  of dimension $7$ such that $ \dim L^{2}=3 $. By looking the classification of all nilpotent Lie algebras in \cite{15}, $L$ must be isomorphic to one of the Lie algebras listed at the Table  $2  $ and $ 3 $.
 \begin{longtable}{cccc}
\multicolumn{2}{c}{\textbf{Table 2.  $ 7 $-dimensional indecomposable nilpotent Lie algebras} }\\
\hline \multicolumn{1}{c}{\textbf{Name}} & \multicolumn{1}{c}{\textsf{Nonzero multiplication}}    \\
\hline
\endhead
\hline \multicolumn{2}{r}{\small \itshape continued on the next page}
\endfoot
\endlastfoot
$ 37A $& $[x_{1},x_{2}] =x_{5}, [x_{2},x_{3}] =x_{6}, [x_{2},x_{4}] =x_{7}$  \\
\\
 $ 37B $  & $[x_{1},x_{2}] =x_{5}, [x_{2},x_{3}] =x_{6}, [x_{3},x_{4}] =x_{7}$  \\
\\
$ 37C $  & $[x_{1},x_{2}] =[x_{3},x_{4}]=x_{5}, [x_{2},x_{3}] =x_{6}, [x_{2},x_{4}] =x_{7}$   \\
\\
$ 37D $  & $[x_{1},x_{2}] =[x_{3},x_{4}]=x_{5}, [x_{1},x_{3}] =x_{6}, [x_{2},x_{4}] =x_{7}$   \\
\\
$ 257A $&$[x_{1},x_{2}] =x_{3}, [x_{1},x_{3}] =[x_{2},x_{4}]=x_{6}, [x_{1}, x_{5}]=x_{7}$\\
\\
$ 257B $&$[x_{1},x_{2}] =x_{3}, [x_{1},x_{3}] =x_{6}, [x_{1}, x_{4}]=[x_{2},x_{5}]=x_{7}$  \\
\\
$ 257C $ & $[x_{1},x_{2}] =x_{3}, [x_{1},x_{3}] =[x_{2}, x_{4}]=x_{6}, [x_{2},x_{5}]=x_{7}$  \\
\\
$ 257D $&$[x_{1},x_{2}] =x_{3}, [x_{1},x_{3}] =[x_{2}, x_{4}]=x_{6},$ 
\\
&$[x_{1},x_{4}]=[x_{2},x_{5}]=x_{7}$  \\
\\
$ 257E $& $[x_{1},x_{2}] =x_{3}, [x_{1},x_{3}] =[x_{4}, x_{5}]=x_{6}, [x_{2},x_{4}]=x_{7}$ \\
\\
$ 257F $& $[x_{1},x_{2}] =x_{3}, [x_{2},x_{3}] =[x_{4}, x_{5}]=x_{6}, [x_{2},x_{4}]=x_{7}$ \\
\\
$ 257G $& $[x_{1},x_{2}] =x_{3}, [x_{1},x_{3}] =[x_{4}, x_{5}]=x_{6},$\\ 
&$[x_{1},x_{5}]=[x_{2},x_{4}]=x_{7}$ \\
\\
$ 257H $& $[x_{1},x_{2}] =x_{3}, [x_{1},x_{3}] =[x_{2}, x_{4}]=x_{6}, [x_{4},x_{5}]=x_{7}$\\
\\
$ 257I $&$[x_{1},x_{2}] =x_{3}, [x_{1},x_{3}] =[x_{1}, x_{4}]=x_{6},$\\ 
&$[x_{1},x_{5}]=[x_{2},x_{3}]=x_{7}$ \\
\\
$ 257J $&$[x_{1},x_{2}] =x_{3}, [x_{1},x_{3}] =[x_{2}, x_{4}]=x_{6},$
$[x_{1},x_{5}]=[x_{2},x_{3}]=x_{7}$\\
\\
$ 257K $&$[x_{1},x_{2}] =x_{3}, [x_{1},x_{3}] =x_{6}, [x_{2},x_{3}] =[x_{4},x_{5}]=x_{7}$\\
\\
$ 257L $&$[x_{1},x_{2}] =x_{3}, [x_{1},x_{3}] =[x_{2}, x_{4}]=x_{6},$\\ 
&$[x_{2},x_{3}]=[x_{4},x_{5}]=x_{7}$\\
\\
 $ 147A $&  $[x_{1},x_{2}] =x_{4},[x_{1},x_{3}]  =x_{5}, $ \\&$[x_{1},x_{6}] =[x_{2},x_{5}]=  [x_{3},x_{4}] =x_{7}$\\
 \\
  $ 147B $&  $[x_{1},x_{2}] =x_{4},[x_{1},x_{3}]  =x_{5}, $\\&$[x_{1},x_{4}] =[x_{2},x_{6}]=  [x_{3},x_{5}] =x_{7}$\\
 \\
  $ 1457A $&  $[x_{1},x_{i}] =x_{i+1}~~ i=2,3 , ~~[x_{1},x_{4}] =[x_{5},x_{6}]=x_{7}$ \\
 \\
 $ 1457B $ &  $[x_{1},x_{i}] =x_{i+1}~~ i=2,3 , $ \\&$[x_{1},x_{4}] =[x_{2},x_{3}]=[x_{5},x_{6}]=x_{7}$\\
 \\
 \\
$ 137A $ & $[x_{1},x_{2}] =x_{5}, [x_{1},x_{5}] =[x_{3},x_{6}] = x_{7}, [x_{3},x_{4}] =x_{6}$   \\
 \\
$ 137B $ & $[x_{1},x_{2}] =x_{5}, [x_{3},x_{4}] =x_{6},$  
  \\& $[x_{1},x_{5}] =[x_{2},x_{4}] = [x_{3},x_{6}] =x_{7}$ \\
 \\
$ 137C $  & $[x_{1},x_{2}] =x_{5}, [x_{1},x_{4}] =[x_{2},x_{3}] =x_{6},$  
  \\& $[x_{1},x_{6}] =x_{7},  [x_{3},x_{5}] =-x_{7}$ \\
 \\
 $ 137D $ & $[x_{1},x_{2}] =x_{5}, [x_{1},x_{4}] =[x_{2},x_{3}] =x_{6},$
  \\& $[x_{1},x_{6}] =[x_{2},x_{4}] =x_{7},  [x_{3},x_{5}] =-x_{7}$ \\
 \\
$ 1357A $& $[x_{1},x_{2}] =x_{4}, [x_{1},x_{4}] =[x_{2},x_{3}] =x_{5},$  
  \\& $[x_{1},x_{5}] =[x_{2},x_{6}] =x_{7},  [x_{3},x_{4}] =-x_{7}$ \\
 \\
 $ 1357B $ & $[x_{1},x_{2}] =x_{4}, [x_{1},x_{4}] =[x_{2},x_{3}] =x_{5},$  
  \\&$[x_{1},x_{5}] =[x_{3},x_{6}] =x_{7},  [x_{3},x_{4}] =-x_{7}$ \\
 \\
 $ 1357C $ & $[x_{1},x_{2}] =x_{4}, [x_{1},x_{4}] =[x_{2},x_{3}] =x_{5},$ 
  \\& $[x_{1},x_{5}] =[x_{2},x_{4}] =x_{7},  [x_{3},x_{4}] =-x_{7}$ \\
 \hline
\end{longtable}
 \begin{longtable}{ccccc| |cccccc}
\multicolumn{3}{c}{\textbf{Table 3.  $ 7 $-dimensional decomposable nilpotent Lie algebras} }\\
\hline \multicolumn{2}{c}{\textbf{Name}} & \multicolumn{2}{c}{\textbf{Name}}   \\
\hline
\endhead
\hline \multicolumn{3}{r}{\small \itshape continued on the next page}
\endfoot
\endlastfoot
&$L_{4,3} \oplus H(1)$ & $ L_{6,19}(\epsilon) \oplus A(1)$ \\
&$ L_{5,6} \oplus A(2)$ &$ L_{6,20} \oplus A(1)$\\
&$ L_{5,7} \oplus A(2)$&$ L_{6,23} \oplus A(1)$\\
&$ L_{5,9} \oplus A(2)$&$ L_{6,24}(\epsilon)\oplus A(1)$\\
&$ L_{6,11} \oplus A(1)$& $ L_{6,25} \oplus A(1)$\\
&$ L_{6,12} \oplus A(1)$&$ L_{6,26} \oplus A(1)$\\
&$ L_{6,13} \oplus A(1)$ \\
 \hline
\end{longtable}
We need the following lemma from \cite[Lemma 2.7]{29-} for the proof of the Main Theorem.
 \begin{lem}\label{1**}
 Let $ L $ be an $ n $-dimensional nilpotent Lie algebra such that $ n=5, 6 $ or $ 7 $, $ \dim L^{2}=\dim Z(L)=3 $ and $ Z(L)=L^{2} $. Then the structure and the Schur multiplier of $L$  are given in the following table.
\begin{longtable}{cccccc}
\multicolumn{3}{c}{\textbf{Table 4. } }\\
\hline \multicolumn{1}{c}{\textsf{Name}}& \multicolumn{1}{c}{\textbf{$\dim\mathcal{M}(L)$}} & \multicolumn{1}{c}{\textsf{$s(L)$}} &    \multicolumn{1}{c}{\textsf{Name}}& \multicolumn{1}{c}{\textbf{$\dim\mathcal{M}(L)$}} & \multicolumn{1}{c}{\textsf{$s(L)$}}    \\
\hline
\endhead
\hline \multicolumn{4}{r}{\small \itshape continued on the next page}
\endfoot
\endlastfoot
 $ L_{6,26} $& $ 8 $  & $ 3 $ &$ 37C $&$ 11 $ & $ 5 $  \\\\
$ 37A $&$ 12 $  & $ 4 $ & $ 37D $&$ 11 $ & $ 5 $   \\\\
$ 37B $ &$11 $ & $ 5 $    \\
 \hline
\end{longtable}
\end{lem}
 \begin{lem}\label{11**}
 Let $ L $ be a nilpotent Lie algebra  of dimension at most $7$ such that $ \dim L^{2}=3 $, $ \dim Z(L)=2 $ and $ Z(L)\subset L^{2} $. Then the structure and the Schur multiplier of $L$ are given in the following table.
 \begin{longtable}{cccccc}
\multicolumn{3}{c}{\textbf{Table 5.}}\\
\hline \multicolumn{1}{c}{\textsf{Name}} &\multicolumn{1}{c}{\textbf{$\dim \mathcal{M}(L)$}} & \multicolumn{1}{c}{\textsf{$s(L)$}} &\multicolumn{1}{c}{\textsf{Name}} &\multicolumn{1}{c}{\textbf{$\dim \mathcal{M}(L)$}} & \multicolumn{1}{c}{\textsf{$s(L)$}}     \\
\hline
\endhead
\hline \multicolumn{3}{r}{\small \itshape continued on the next page}
\endfoot
\endlastfoot
$ L_{5,9} $ & $ 3 $  & $ 4 $  & $ 257E $ & $ 8 $ & $ 8 $ \\ \\\\
$ L_{6,23} $&$ 6$ & $ 5 $ & $ 257 F$ & $ 9 $ & $ 7 $ \\\\\\
$ L_{6,24}(\epsilon) $&$ 5$  & $ 6 $ &$ 257G $ & $ 8 $ & $ 8 $ \\\\\
$ L_{6,25} $&$ 6 $ & $ 5 $ &$ 257 H$ & $ 8 $ & $ 8 $ \\ \\\\
$ 257A $ & $ 9 $ & $ 7 $ &$ 257I $ & $ 8 $ & $ 8 $ \\ \\\\
$ 257B $ & $ 8 $ & $ 8 $ &$ 257J $ & $ 8 $ & $ 8 $ \\\\\\
$ 257C $ & $ 9 $ & $ 7 $&$ 257K $ & $ 6 $ & $ 10 $ \\ \\\\
$ 257D $ & $ 8 $ & $ 8 $ & $ 257L $ & $ 6 $ & $ 10 $ \\\\\\
  \hline
\end{longtable}
\begin{proof}
The proof is similar to \cite[Lemma 2.5]{29-}. 
\end{proof}
 \end{lem}
\begin{lem}\label{L1}
 Let $ L $ be an $n$-dimensional nilpotent Lie algebra such that $ n=7, $  $ \dim L^{2}=3 $ and $ \dim Z(L)=4 $. Then the structure and the Schur multiplier of $L$ are given in the following table.
 \begin{longtable}{ccc }
\multicolumn{3}{c}{\textbf{Table 6.}}\\
\hline \multicolumn{1}{c}{\textsf{Name}} &\multicolumn{1}{c}{\textbf{$\dim \mathcal{M}(L)$}} & \multicolumn{1}{c}{\textsf{$s(L)$}}  \\
\hline
\endhead
\hline \multicolumn{3}{r}{\small \itshape continued on the next page}
\endfoot
\endlastfoot
$ L_{5,9}\oplus A(2) $& $ 8 $ & $ 8 $  \\
$ L_{6,26}\oplus A(1) $ & $ 11$  & $5 $     
\\
  \hline
\end{longtable}
\begin{proof}
Since $ \dim Z(L)=4, $ $ L $ is isomorphic to $ L_{5,9}\oplus A(2) $ or $ L_{6,26}\oplus A(1) $ by searching in Tables $ 2 $ and $ 3. $ Let $ L\cong  L_{6,26}\oplus A(1). $ Since $ \dim \mathcal{M}(L_{6,26})=8 $ by using Table $ 4, $ we have  $\dim \mathcal{M}(L)=11$ by using \cite[Theorem 1]{2} and \cite[Lemma 23]{16}. Hence $s(L)=5.$ Also by using similar method , we can see $\dim \mathcal{M}(L_{5,9}\oplus A(2))=8$ and $ s(L)=8. $
\end{proof}
\end{lem} 
 \begin{lem}\cite[Lemma 2.9]{29-}\label{111**}
 Let $ L $ be an $n$-dimensional nilpotent Lie algebra such that $ n=5, 6 $ or $ 7 $, $ \dim L^{2}=3 $ and $ \dim Z(L)=1 $. Then the structure and the Schur multiplier of $L$ are given in the following table.
 \begin{longtable}{ccc ccc}
\multicolumn{3}{c}{\textbf{Table 7.}}\\
\hline \multicolumn{1}{c}{\textsf{Name}} &\multicolumn{1}{c}{\textbf{$\dim \mathcal{M}(L)$}} & \multicolumn{1}{c}{\textsf{$s(L)$}} &\multicolumn{1}{c}{\textsf{Name}} &\multicolumn{1}{c}{\textbf{$\dim \mathcal{M}(L)$}} & \multicolumn{1}{c}{\textsf{$s(L)$}} \\
\hline
\endhead
\hline \multicolumn{3}{r}{\small \itshape continued on the next page}
\endfoot
\endlastfoot
$ L_{5,6} $& $ 3 $ & $ 4 $  &$ 1457A $& $ 6$  & $10$  \\\\
$ L_{5,7} $&$ 3$  & $ 4 $ &  $ 1457B $&$ 6 $  & $10$   \\ \\
$ L_{6,11}$&$ 5$ & $ 6 $ &$ 137A $& $ 7 $  & $ 9 $   \\\\
$ L_{6,12} $&$ 5 $ & $6 $ &$ 137B $&  $ 7 $  & $ 9 $       \\\\
 $ L_{6,13} $&$ 4 $  & $ 7$ & $ 137C $& $ 7 $  & $ 9 $   \\\\
$ L_{6,19}(\epsilon) $& $ 5 $ & $ 6 $ &   $ 137D $&$ 7 $ & $ 9$   \\\\
 $ L_{6,20} $&$ 5 $  & $ 6 $ &$ 1357A $& $ 7 $ & $ 9 $     \\\\
 $ 147A $&$ 8 $  & $8$ &  $ 1357B $&$ 6 $ & $ 10 $ \\\\
 $ 147B $&$ 8 $  & $8$ &  $ 1357C $&$6$  & $ 10 $  \\
  \hline
\end{longtable}
 \end{lem}
Recall that a Lie algebra $L$ is  called generalized Heisenberg of rank $n$ if $L^2 = Z(L)$ and $\dim L^2 = n$.
 \begin{lem}\label{27}
 Let $ L $ be an $ n $-dimensional generalized Heisenberg of rank $ 3 $ with $ s(L)=5 $, then $ n\leq 7 $.
 \begin{proof}
 By Theorem \ref{16} we have $ \dim \ker g= \dim \mathcal{M}(L^{ab})-\dim L^{2}+\dim L^{ab}\otimes_{mod} L^{2}-\dim \mathcal{M}(L) $. Since $ \dim \mathcal{M}(L)= \frac{1}{2}(n-1)(n-2)-4 $ and $ \dim L^{ab}=n-3 $, we have $ \dim \ker g=n-3 $.\\
By contrary let  $ n\geq 8 $. Then $ d=\dim L^{ab}=n-3\geq 5$. Since $ \dim L^{2}=3 $, we can choose  a basis $ \lbrace x_{1}+L^{2}, ...,  x_{d}+L^{2}\rbrace$ for $ L^{ab} $ such that $ [x_{1}, x_{2}] $, $  [x_{2}, x_{3}] $ and $  [x_{3}, x_{4}] $ are non-trivial in $ L^{2} $. Thus 
 \begin{equation*}
 L^{ab}\otimes_{mod} L^{2}\cong \bigoplus \limits_{i=1}^d (\langle x_{i}+ L^{2}) \rangle\otimes_{mod}L^{2}).
 \end{equation*}
 Hence all elements of 
\begin{equation*}
\lbrace [x_{1}, x_{2}]\otimes x_{i}+L^{2}\oplus [x_{i}, x_{1}]\otimes x_{2}+L^{2}\oplus [x_{2}, x_{i}]\otimes x_{1}+L^{2},\mid 3\leq i \leq d, i\neq 1, 2 \rbrace
\end{equation*}
and
\begin{equation*}
\lbrace [x_{2}, x_{3}]\otimes x_{i}+L^{2}\oplus [x_{i}, x_{3}]\otimes x_{2}+L^{2}\oplus [x_{3}, x_{i}]\otimes x_{2}+L^{2},\mid 3\leq i \leq d, i\neq 1, 2, 3 \rbrace 
\end{equation*}
\begin{equation*}
\lbrace [x_{3}, x_{4}]\otimes x_{i}+L^{2}\oplus [x_{i}, x_{3}]\otimes x_{4}+L^{2}\oplus [x_{4}, x_{i}]\otimes x_{3}+L^{2},\mid 3\leq i \leq d, i\neq  2,3,4 \rbrace
\end{equation*}
are linearly independent and so $ 2(n-6)+n-5 \leq \ker g $. That is a contradiction for $ n\geq  8$. Therefore the assumption is false and the result follows. 
 \end{proof}
 \end{lem}
Let $c(L)$ be used to show the nilpotency class of $L$. Then
  \begin{lem}\label{28}
 There is no  nilpotent Lie algebra $ L $ with $ \dim L^{2}=3 $,  $ \dim Z(L)=1 $ and $ s(L)=5 $ such that $ L/ Z(L)\cong L_{5,8}\oplus A(2) $.
 \begin{proof}
 By contrary, let $ L $ be a nilpotent Lie algebra $ L $ with $ \dim L^{2}=3 $,  $ \dim Z(L)=1 $ and $ s(L)=5 $ such that $ L/ Z(L)\cong L_{5,8}\oplus A(2) $. Then $ \dim L=8 $ and $ cl(L)=3 $. Since $ cl(L)=3 $ and $ \dim Z(L)=1$, we have $ L^{3}=Z(L) $. On the other hand, $  \dim \mathcal{M}(L) =\dim \mathcal{M}(L/Z(L))+(\dim L/L^{2}-1)\dim Z(L)-\dim \ker \lambda _{3}  $ and   $ \dim \ker \lambda _{3}\geq 2 $  by using proof \cite[Theorem 1.1]{27-}. Thus  

 $\dim \mathcal{M}(L) \leq \dim \mathcal{M}(L/Z(L))+(\dim L/L^{2}-1)\dim Z(L)-2$

 It is a contradiction.
\end{proof}
 \end{lem}
 \begin{thm}
  Let $L$ be an $n$-dimensional nilpotent Lie algebra with $s(L)=5$ and $\dim L^{2}=3$. Then  $ L $ is isomorphic to one of the Lie algebras $L_{6,23}$,  $ L_{6,25} $, $ 37B $, $ 37C $ or $ 37D $.
 \begin{proof}
 First assume that $\dim Z(L)\geq 5$, or $\dim Z(L)=3$ and $Z(L)\neq L^{2}$, or $\dim Z(L)=2 $  and $Z(L)\not\subset L^{2}$. We show that in these cases,  there is no such a Lie algebra $L$ of dimension $n$ with $ s(L)=5 $.\\
  Let   $I$ be a central ideal of $L$ of dimension one such that $L^{2}\cap I=0$. Since $\dim (L/I)^{2}=3$, by using \cite[Theorem 3.1]{25}, we have
 \begin{align*}
 \dim \mathcal{M}(L/I)\leq \frac{1}{2}n(n-5)+1.
 \end{align*}
 If the equality holds, then
\begin{align*}
\frac{1}{2}(n-2)(n-3)+1-s(L/I)=\dim \mathcal{M}(L/I)= \frac{1}{2}n(n-5)+1.	
\end{align*}
Therefore $ s(L/I)=3 $ and by using \cite[Theorem 3.2 ]{29}, there is no Lie algebra satisfying in $ \dim (L/I)^{2}=3 $. Thus \cite[corollary 2.3]{25} and our assumption imply
\begin{equation*}
\dim \mathcal{M}(L)=\frac{1}{2}(n-1)(n-2)-4\leq \frac{1}{2}n(n-5)+(n-4),	
\end{equation*}
 which is a contradiction. Therefore we may assume that $ \dim Z(L)=4, $ or $\dim Z(L)=3$ and   $L^{2}=Z(L)$, or $\dim Z(L)=2$ and $ Z(L)\subset L^{2} $, or $ \dim Z(L)=1 $.\\
 If $\dim Z(L)=4, $ then there is a central ideal of $L$ of dimension one such that $L^{2}\cap I=0$. Since $\dim (L/I)^{2}=3$, by using \cite[Theorem 3.1]{25}, we have
 \begin{align*}
 \dim \mathcal{M}(L/I)\leq \frac{1}{2}n(n-5)+1.
 \end{align*}
 If the equality holds, then
\begin{align*}
\frac{1}{2}(n-2)(n-3)+1-s(L/I)=\dim \mathcal{M}(L/I)= \frac{1}{2}n(n-5)+1.	
\end{align*}
Therefore $ s(L/I)=3 $ and by using Table $ 4, $ $ L/I\cong L_{6,26}. $ Since $\dim Z(L)=4 $ and $ \dim L=7, $ we have $ L\cong L_{6,26}\oplus A(1) $ by using Lemma \ref{L1}. Now let $ \dim \mathcal{M}(L)\leq \frac{1}{2}n(n-5). $
Thus \cite[corollary 2.3]{25} and our assumption imply
\begin{equation*}
\dim \mathcal{M}(L)=\frac{1}{2}(n-1)(n-2)-4\leq \frac{1}{2}n(n-5)+(n-4),	
\end{equation*}
 which is a contradiction. \\
If $\dim Z(L)=3$ and $L^{2}=Z(L)$, then $ L $ is isomorphic to one of  the Lie algebras $ 37B $, $ 37C $ or $ 37D $ by using  Lemmas \ref{1**} and \ref{27}.\\
Assume that $\dim Z(L)=2$ and $Z(L)\subset L^{2}$. Then $\dim (L/Z(L))^{2}=1$. Since $L/Z(L) $ capable, by using \cite[Theorem 3.5]{24} and \cite[Lemma 3.3]{17}, we have  $L/Z(L)\cong H(1)\oplus A(n-5)$. Hence  $ L $ is nilpotent of class $ 3 $. Therefore, by using \cite [Theorem 2.6]{29} for $c=3$, we have \[\dim L^3+\frac{1}{2}(n-1)(n-2)-3\leq \dim \mathcal{M}(L/L^3)+\dim (L/Z_2(L)\otimes L^3).\] Now since $1\leq \dim L^3\leq 2$, we can obtain that
$ n\leq 7 $.
Hence  Lemma \ref{11**} implies that $ L\cong L_{6,23}$ or $L\cong L_{6,25} $.
\\Finally, assume that $\dim Z(L)=1$. Then $\dim (L/Z(L))^{2}=2$. By using \cite[corollary 2.3]{25}, we have 
\begin{equation*}
\frac{1}{2}(n-1)(n-2)-3\leq \frac{1}{2}(n-2)(n-3)+1-s(L/Z(L))+n-3.
\end{equation*}
Thus $ s(L/Z(L))\leq 3 $.\\
If $ s(L/Z(L))=0 $, then $ L\cong H(1)\oplus A(n-4) $ by  \cite[Theorem 3.1]{25}. This case cannot occur, since $\dim (L/Z(L))^{2}=2$.\\
If $ s(L/Z(L))=1 $, then \cite[Theorem 3.9]{17} implies that $ L\cong L(4,5,2,4) $. Therefore $ n=6 $.\\
If $ s(L/Z(L))=2 $, then $ L/Z(L) $ is  isomorphic to one of the Lie algebras $L(3,4,1,4)$, $L(4,5,2,4)$ $\oplus A(1) $ or $ H(m)\oplus A(n-2m-1) (m\geq 2)$ by using \cite[Theorem 4.5]{17}. 
In the case $L(3,4,1,4)$ or $ L(4,5,2,4)$ $\oplus A(1) $, we have $ n=5 $ or $ 7 $.\\
In the case $ L/Z(L)\cong  H(m)\oplus A(n-2m-1) (m\geq 2)$, then we have a contradiction, since $ \dim (L/Z(L))^{2}=2 $.\\
If $ s(L/Z(L))=3 $,  $ L/Z(L) $  is  isomorphic to one of the Lie algebras   $ L(4,5,1,6)$, $L(5,6,2,7) $, $ L^{\prime}(5,6,2,7) $, $ L(7,6,2,7) $, $ L^{\prime}(7,6,2,7) $ or $ L(3,4,1,4)\oplus A(1) $  by \cite[Main Theorem]{29} and Lemma \ref{28}.\\
Hence $ n=5,6 $ or $ 7 $ when $ \dim Z(L)=1 $. But there is no such Lie algebra by Lemma \ref{111**}. This completes proof.
\end{proof}
 \end{thm}

\end{document}